\documentclass[11pt,a4paper,reqno]{amsproc}

\makeatletter

\usepackage{datetime}
\usepackage[pdfborder={0 0 0}]{hyperref}
  \hypersetup{
    citebordercolor=.1 .6 1,
    colorlinks = false
}

\usepackage{upref,cases}
\hypersetup{citebordercolor=.1 .6 1}

\usepackage[T1]{fontenc}
\usepackage[utf8]{inputenc}

\usepackage[final]{graphicx}

\renewcommand*\env@matrix[1][*\c@MaxMatrixCols c]{%
  \hskip -\arraycolsep
  \let\@ifnextchar\new@ifnextchar
  \array{#1}}

\IfFileExists{mmymtpro2.sty}{\usepackage[subscriptcorrection]{mymtpro2}

  \def\cup{\cupprod}
  
  \def\bigcup{\bigcupprod}
  
  \def\bigcupdisjoint{\mathop{\kern10pt\raisebox{4pt}{$\cdot$}\kern-12pt\bigcup}\limits}
}%
	{\usepackage{amssymb,bm,dsfont,fourier,times}

	\let\wtilde\widetilde
}

\let\epsilon\varepsilon

\DeclareMathOperator{\supp}{supp}
\DeclareMathOperator{\dist}{dist}

\numberwithin{equation}{section}
\newtheoremstyle{ttheorem}%
       {1.8ex\@plus1ex}                
       {2.1ex\@plus1ex\@minus.5ex}      
       {\itshape}           
       {0pt}                   
       {\bfseries}          
       {.}                  
       {.5em}               
       {}                

\newtheoremstyle{ddefinition}%
       {1.8ex\@plus1ex}                
       {2.1ex\@plus1ex\@minus.5ex}      
       {}           
       {0pt}                   
       {\bfseries}           
       {.}                  
       {.5em}               
       {}                

\newtheoremstyle{rremark}%
       {1.8ex\@plus1ex}                
       {2.1ex\@plus1ex\@minus.5ex}      
       {\normalfont}        
       {0pt}                   
       {\bfseries}           
       {.}                  
       {.5em}               
       {}                   

\theoremstyle{ttheorem}
\newtheorem{theorem}{Theorem}[section]
\newtheorem{lemma}[theorem]{Lemma}
\newtheorem{proposition}[theorem]{Proposition}
\newtheorem{corollary}[theorem]{Corollary}

\theoremstyle{ddefinition}
\newtheorem{definition}[theorem]{Definition}

\theoremstyle{rremark}

\newtheorem{myremarks}[theorem]{Remarks}
\newtheorem{myexamples}[theorem]{Examples}
\newtheorem{example}[theorem]{Example}

\newenvironment{remarks}{\begin{myremarks}\begin{nummer}}%
    {\end{nummer}\end{myremarks}}
    {\end{nummer}\end{myexamples}}

\newcounter{numcount}
\newcommand{\labelnummer}{(\roman{numcount})}%

\makeatletter
\providecommand{\showkeyslabelformat}[1]{\relax}        
\let\mysaveformat\showkeyslabelformat                   %
\def\myformat#1{\raisebox{-1.5ex}{\mysaveformat{#1}}}   %

\newenvironment{nummer}%
  {\let\curlabelspeicher\@currentlabel%
    \begin{list}{\textup{\labelnummer}}%
      {\usecounter{numcount}\leftmargin0pt%
        \topsep0.5ex\partopsep2ex\parsep0pt\itemsep0ex\@plus1\p@%
        \labelwidth2.5em\itemindent3.5em\labelsep1em%
      }%
    \let\saveitem\item%
    \def\item{\saveitem%
      \def\@currentlabel{\curlabelspeicher\kern.1em\labelnummer}}%
    \let\savelabel\label%
    \def\label##1{{\ifnum\thenumcount=1\let\showkeyslabelformat\myformat\fi\savelabel{##1}}%
										{\def\@currentlabel{\labelnummer}%
									 	\let\showkeyslabelformat\@gobble
									 	\savelabel{##1item}%
										}%
	   							}%
  }{\end{list}}%

  {\let\curlabelspeicher\@currentlabel%
    \begin{list}{\textup{\labelnummer}}%
      {\usecounter{numcount}\leftmargin0pt%
        \topsep0.5ex\partopsep2ex\parsep0pt\itemsep0ex\@plus1\p@%
        \labelwidth2.5em\itemindent0em\labelsep1em%
        \leftmargin2.5em}%
    \let\saveitem\item%
    \def\item{\saveitem%
      \def\@currentlabel{\curlabelspeicher\kern.1em\labelnummer}}%
    \let\savelabel\label%
    \def\label##1{{\ifnum\thenumcount=1\let\showkeyslabelformat\myformat\fi\savelabel{##1}}%
										{\def\@currentlabel{\labelnummer}%
									 	\let\showkeyslabelformat\@gobble
									 	\savelabel{##1item}%
										}%
    							}%
  }{\end{list}}%

\usepackage{enumitem}

\let\OldItem\item
\newcommand{\MyItem}[2][]{}%
\def\section{\@startsection{section}{1}%
  \z@{1.3\linespacing\@plus\linespacing}{.5\linespacing}%
  {\normalfont\bfseries\centering}}

\def\subsection{\@startsection{subsection}{2}%
  \z@{.8\linespacing\@plus.5\linespacing}{-1em}%
  {\normalfont\bfseries}}

\def\nlsubsection{\@startsection{subsection}{2}%
  \z@{.8\linespacing\@plus.5\linespacing}{.1ex}%
  {\normalfont\bfseries}}

\let\@afterindenttrue\@afterindentfalse%

\renewenvironment{proof}[1][\proofname]{\par \normalfont
  \topsep6\p@\@plus6\p@ \trivlist 
  \item[\hskip\labelsep\scshape
    #1\@addpunct{.}]\ignorespaces
}{%
  \qed\endtrivlist
}

\def\ps@firstpage{\ps@plain
  \def\@oddfoot{\normalfont\scriptsize \hfil\thepage\hfil
     \global\topskip\normaltopskip}%
  \let\@evenfoot\@oddfoot
  \def\@oddhead{
    \begin{minipage}{\textwidth}
      \normalfont\scriptsize
      \emph{\insertfirsthead}
    \end{minipage}}
  \let\@evenhead\@oddhead 
}

\def\insertfirsthead{}

\def\@cite#1#2{{%
 \m@th\upshape\mdseries[{#1}{\if@tempswa, #2\fi}]}}
\newcommand{\C}{\mathbb{C}}
\newcommand{\N}{\mathbb{N}}

\newcommand{\R}{\mathbb{R}}
\newcommand{\Z}{\mathbb{Z}}

\renewcommand{\leq}{\leqslant}
\renewcommand{\geq}{\geqslant}
\providecommand{\wtilde}[1]{\widetilde{#1}}

\let\<\langle
\let\>\rangle

\newcommand{\upd}{\mathrm{d}}
\renewcommand{\d}{\upd}   

\newcommand{\hairspace}{\kern .04167em}

\newcommand{\beq}{\begin{equation}}
\newcommand{\eeq}{\end{equation}}

\makeatother

\begin{document}

\title[Dirichlet-Neumann bracketing for Toeplitz matrices]{Dirichlet-Neumann bracketing for a class of  banded Toeplitz matrices}

\author[M.\ Gebert]{Martin Gebert}
\address[M.\ Gebert]{Mathematisches Institut,
  Ludwig-Maximilians-Universit\"at M\"unchen,
  Theresienstra\ss{e} 39,
  80333 M\"unchen, Germany}

\email{gebert@math.lmu.de}

\maketitle

\begin{abstract}
We consider boundary conditions of self-adjoint banded Toeplitz matrices.
We ask if boundary conditions exist for banded self-adjoint Toeplitz matrices which satisfy operator inequalities of Dirichlet-Neumann bracketing type. For a special class of banded Toeplitz matrices including integer powers of the discrete Laplacian we find such boundary conditions. Moreover, for this class we give a lower bound on the spectral gap above the lowest eigenvalue. 
\end{abstract}

\section{Introduction and result}
In this note we are concerned with self-adjoint banded Toeplitz matrices. Let $\mathbb T := (0,2\pi]$ and $L\in\N$. We consider symbols of the form
\beq\label{symbol}
	f:\mathbb T\to\R, \quad f(x) = \sum_{k=-N}^N a_k e^{-i kx}
\eeq
for some $N\in\N$, $a_k\in\C$ with $a_k = \overline{a}_{-k}\in\C$ for $k=-N,...,N$.
These give rise to self-adjoint banded Toeplitz matrices given by the sequence $...,0,a_{-N},...,a_0,...,a_N,0,...$
and $T_{f,L}$ is the corresponding $L\times L$ Toeplitz matrix
\beq
	T_{f,L} =
	 \begin{pmatrix}
		a_0 & a_1 & \cdots &a_N & &    &  & \\
		&\ddots	&  & &\ddots  & &\\
		& a_{-N} & \cdots & a_0 & \cdots & a_N \\
		& &\ddots	&  & &\ddots  & &\\
		 & &  &    a_{-N} & \cdots & a_{-1}  & a_{0} 
	\end{pmatrix}.
\eeq
 Throughout we assume that the matrix size $L$ is bigger than the band width $2N+1$. 
Moreover, $T_f$ stands for
the so-called Laurent or bi-infinite Toeplitz matrix 
\beq\label{def:Laurent}
T_f:\ell^2(\Z)\to\ell^2(\Z),\quad 	(T_f b)_n := \sum_{m\in\Z} a_{m-n} b_m
\eeq
where $b = \big( b_m\big)_{m\in\Z}\in \ell^2(\Z)$
and
$a_k =\displaystyle \frac 1 {2\pi} \int_{0}^{2\pi} \d x\, f(x) e^{-i k x}$, $k\in\Z$. We write $(T_f)_{[a,b]}$ for the restriction of $T_f$ to $\ell^2([a,b])\subset \ell^2(\Z)$ for $a,b\in\Z$ with $a<b$. Then $T_{f,b-a+1}$ is the same matrix as $(T_f)_{[a,b]}$ and we use both notations interchangeably. 
 For further reading about banded Toeplitz matrices we refer to \cite{MR2179973}.

We brake $T_{f,L}$ into the direct sum of two Toeplitz matrices
\beq
	T_{f,L_1}\oplus T_{f,L_2} = 
	\begin{pmatrix}[cccc|cccc]
		a_0 & \cdots & a_N &    &  0 & & &			\\
		& \ddots & \ddots  & &\vdots  & &  &				\\
		 & a_{-N} & \cdots   & a_{0} & 0  &\cdots  & \cdots & 0 		\\
		 \hline
		0& \cdots &\cdots  & 0 & a_0 & \cdots & a_N &    			\\
		& & & \vdots & & \ddots & \ddots  & 				\\
		& & & 0 &  &  a_{-N} & \cdots   & a_{0} 		\\
	\end{pmatrix}
\eeq
where $L_1, L_2\in\N$ with $L_1+ L_2 = L$ and we assume for convenience that $L_1,L_2\geq 2N+1$. It is clear that the difference $T_{f,L} - T_{f,L_1}\oplus T_{f,L_2}$ is of no definite sign and therefore no operator inequality between the two operators $T_{f,L}$ and $T_{f,L_1}\oplus T_{f,L_2}$ holds. 

We are interested in adding boundary conditions to $T_{f,L_1}$ and $T_{f,L_2}$ which overcome this lack of monotonicity.
For a banded Toeplitz matrix with band size $2N+1$ boundary conditions refer to adding Hermitian $N\times N$ matrices at the corner of the respective boundary, i.e. a boundary condition $\star$ is given by a Hermitian $N\times N$ matrix $B_\star$ and 
 \beq\label{def:bc}
	 T^{\star,0}_{f,L} := 
	T_{f,L} +
	\begin{pmatrix}
		B_\star  & 0 \\
		0 & 0
	\end{pmatrix}
	\quad\text{and}\quad
	T^{0,\star}_{f,L} := 
	T_{f,L} +
	\begin{pmatrix}
		0  & 0 \\
		0 & \wtilde B_\star
	\end{pmatrix}
\eeq
where $\wtilde{ B}_\star$ is the reflection of $B_\star$ along the anti-diagonal, i.e. $\wtilde{ B}_\star := U^* B_\star U$ with $U: \C^N\to\C^N$, $ (Ux)_k := x_{N-k+1}$ for $x = (x_1,...,x_N)\in \C^N $. 
The superscript $0$ in the above indicates simple boundary condition at the respective endpoint which refers to no $N\times N$ matrix added. If simple boundary conditions are imposed at both endpoints we drop the $0$ superscripts and note $T^{0,0}_{f,L} = T_{f,L}$.

Our goal is to find boundary conditions $\mathcal N$ and $\mathcal D$ which give rise to a chain of operator inequalities of the form
 \begin{align}\label{intro:D/N}
	T^{0,\mathcal N}_{f,L_1} \oplus T^{\mathcal N,0}_{f,L_2} 
	\leq T_{f,L} \leq  
	T^{0,\mathcal D}_{f,L_1} \oplus T^{\mathcal D,0}_{f,L_2} 
\end{align}
subject to the constraint
\beq\label{intro:D/N:constraint}
	\inf_{\mathbb T} f \leq T^{\mathcal N,\mathcal N}_{f,R} \leq 
	T^{\mathcal D,\mathcal D}_{f,R}  \leq \sup_{\mathbb T} f
\eeq
where $R\in\{L_1,L_2\}$. Inequality \eqref{intro:D/N} is easily satisfied for boundary conditions given by large multiples of the $N\times N$ identity however the non-trivial constraint is \eqref{intro:D/N:constraint} which ensures that the spectra of the restricted operators are subsets of the spectrum of the corresponding infinite-volume operator \eqref{def:Laurent}. We address the question: 

\textit{
Given a banded self-adjoint Toeplitz matrix, do boundary conditions in the sense of \eqref{def:bc} exist such that inequalities \eqref{intro:D/N} and \eqref{intro:D/N:constraint} hold for all $L_1,L_2\in\N$ with $L = L_1 + L_2$ and $L_1,L_2$ greater than the band width?
}

Throughout we mainly focus on the boundary condition $\mathcal N$ and later on find boundary conditions $\mathcal N$ for a special class of banded Topelitz matrices  which satisfy the respective inequalities in \eqref{intro:D/N} and \eqref{intro:D/N:constraint}. We don't know if the answer to the above question remains yes for general banded self-adjoint Toeplitz matrices.

 A chain of inequalities of the form \eqref{intro:D/N} and \eqref{intro:D/N:constraint} is referred to as Dirichlet-Neumann bracketing.
This stems from the following: 
For the continuous negative Laplace operator Dirichlet and Neumann boundary conditions naturally satisfy the operator inequality \eqref{intro:D/N}, see e.g. \cite[Sec. XIII]{MR0493421}. Inspired by the continuous definition, this was later extended to the discrete Laplacian as well \cite[Sec. 5.2]{MR2509110}. In both cases an inequality of the form \eqref{intro:D/N} is by now a standard tool in mathematical physics and was, for example, used in the proof of Lifshitz tails for random Schr\"odinger operators \cite{MR784931,MR2509110,MR2307751} and Weyl asymptotics for continuum Schr\"odinger operators \cite[Sec. XIII]{MR0493421}. 

It might be tempting to think the natural Neumann boundary condition for $T_{f,L}$ satisfies the first inequality in \eqref{intro:D/N}. This boundary condition, which we denote by the superscript $N$, is given by the Toeplitz-plus-Hankel matrix 
\beq\label{Neumann}
	T_{f,L}^{N,0} = T_{f,L} + 
	\begin{pmatrix}
		a_{-1} & \cdots & a_{-N} & \cdots \\
		\vdots & \reflectbox{$\ddots$} \\
		a_{-N} &  &  0 \\
		\vdots & & & \ddots
	\end{pmatrix},
\eeq
see e.g. \cite{MR1718798}. Here we abuse notation a little as the superscript $N$ for Neumann boundary condition has nothing to do with the subscript $N$ indicating the band width of the matrix. 
Except in the case of a self-adjoint $3$-diagonal Toeplitz matrix this boundary condition does not satisfy $T^{0,N}_{L_1,f} \oplus T^{N,0}_{L_2,f} \leq T_{L,f}$. To see this, we consider the square of the negative discrete Laplacian on $\ell^2(\Z)$. Throughout, the negative discrete Laplacian $-\Delta$ is the $3$-diagonal Toeplitz matrix given by the rows $(\cdots,0,-1,2,-1,0,\cdots)$ and therefore $(-\Delta)^2 = \Delta^2$ is $5$-diagonal and given by 
$(\cdots,0,1,-4,6,-4,1,0,\cdots)$. In that case a computation shows that 
\beq\label{Lap2}
	\big(\Delta^2\big)_{L} - \big(\Delta^2\big)^{0,N}_{L_1} \oplus \big(\Delta^2\big)^{N,0}_{L_2}  = 
	\begin{pmatrix}
		\ddots & & & & & \reflectbox{$\ddots$}\\ 
		&0 & -1 & 1 & 0  \\
		&-1 & 4 & -4 & 1 \\
		&1 & -4 & 4 & -1\\
		&0 & 1 & -1 & 0 \\
		\reflectbox{$\ddots$} & & & & & \ddots
	\end{pmatrix}
\eeq
and $0$ everywhere else. This matrix is not of definite sign and therefore the first inequality from the left in \eqref{intro:D/N} does not hold.

In this note we introduce what we call modified Neumann boundary conditions $\mathcal N$ which satisfy the first inequality in \eqref{intro:D/N} and \eqref{intro:D/N:constraint} for Toeplitz matrices given by symbols of the form 
\beq\label{def:symbol}
	f_{E_1,\cdots,E_n,\alpha_1,\cdots,\alpha_n}(x) := f_{E_1}^{\alpha_1}(x)\cdots f_{E_n}^{\alpha_n}(x) = 
	\prod_{i=1}^n \big(2-2\cos(x-E_i)\big)^{\alpha_i}
\eeq
for $x\in \mathbb T$ and some distinct $E_1,...,E_n\in\mathbb T$ and $\alpha_1,...,\alpha_n \in\N$. In the above, we have set
$
f_E^\alpha(x) := \big(2-2\cos(x-E)\big)^{\alpha}.
$
Note that the minimum of $f_{E_1,\cdots,E_n,\alpha_1,\cdots,\alpha_n}$ is $0$ and it is attained at the points $E_1,...,E_n$. 

We remark that $T_{f^1_{E}}$ is a $3$-diagonal Laurent matrix given by rows $(\cdots,0,e^{iE}, 2, e^{-i E},0,\cdots)$ which is unitarily equivalent to the discrete Laplacian $-\Delta$ and we set $-\Delta_E:=T_{f^1_{E}}$. Using this notation, we can write
\beq\label{toeplitz}
	T_{f_{E_1,\cdots,E_n,\alpha_1,\cdots,\alpha_n}} = \prod_{i=1}^n \big(-\Delta_{E_i}\big)^{\alpha_i}
\eeq
which is a banded Laurent matrix with band width $2N+1$ where $N = \sum_{i=1}^n \alpha_i$. 
The main theorem regarding Dirichlet-Neumann bracketing for $T_{f_{E_1,\cdots,E_n,\alpha_1,\cdots,\alpha_n}}$ is the following:

\begin{theorem}\label{thm:main1}
Let $n\in\N$, $E_1,..,E_n\in\mathbb T$ be distinct and $\alpha_1,...,\alpha_n\in\N$. Let $g = f_{E_1,\cdots,E_n,\alpha_1,\cdots,\alpha_n}$ be of the form \eqref{def:symbol} and $N=\sum_{i=1}^n \alpha_i$. Then there exist boundary conditions which we call modified Neumann  and Dirichlet boundary conditions, $\mathcal N$ and $\mathcal D$, such that 
\beq\label{thm:D/N-brack}
	T^{0,\mathcal N}_{g,L_1} \oplus T^{\mathcal N,0}_{g,L_2} 
	\leq T_{g,L} 
	\leq T^{0,\mathcal D}_{g,L_1} \oplus T^{\mathcal D,0}_{g,L_2}
\eeq
and
\beq
	0 = \inf_{\mathbb T} g
	\leq T^{\mathcal N,\mathcal N}_{g,L_1} \oplus T^{\mathcal N,\mathcal N}_{g,L_2} 
	\leq T^{0,\mathcal N}_{g,L_1} \oplus T^{\mathcal N,0}_{g,L_2} 
\eeq
for all $L_1,L_2\in\N$ with $L_1+ L_2 = L$ and $L_1,L_2\geq 2N+1$.
The boundary conditions $\mathcal N$ and $\mathcal D$ are given explicitly in Definition \ref{def:Neumann} below.
\end{theorem}

\begin{remarks}
\item 
For band width greater than $3$ the boundary conditions $\mathcal N$ and $\mathcal D$ differ from the Neumann boundary N condition mentioned in \eqref{Neumann} and the Dirichlet boundary condition used in e.g. \cite{MR1718798} which coincides with what we call simple boundary condition.
\item 
It would be desirable to have the inequality $T^{0,\mathcal D}_{f,L_1} \oplus T^{\mathcal D,0}_{f,L_2} \leq \sup_{\mathbb T} f$ as well but our modified Dirichlet boundary condition $\mathcal D$ defined in  Definition \ref{def:Neumann} does not satisfy this. We obtain $\mathcal D$ by a general principle  that any inequality $T^{0,\mathcal N}_{g,L_1} \oplus T^{\mathcal N,0}_{g,L_2} \leq T_{g,L}$ induces  modified Dirichlet boundary condition such that $ T_{g,L} \leq T^{0,\mathcal D}_{g,L_1} \oplus T^{\mathcal D,0}_{g,L_2}$ and vice versa, see Lemma \ref{Neu=Dir}. 
\item 
The theorem holds for any integer power ($m\in\N$) of the discrete Laplacian as the symbol of $(-\Delta)^m:\ell^2(\Z)\to\ell^2(\Z)$ is 
\beq	
	g(x) = f_{0,m}(x) = \big(2-2\cos(x)\big)^m,\quad x\in\mathbb T
\eeq
and is of the form \eqref{def:symbol}. In that case $n=1$, $E_1 =0$ and  $\alpha_1= m$. 
\end{remarks}

Considering only symbols \eqref{def:symbol} seems very restrictive. But, for example, Theorem \ref{thm:main1} gives Dirichlet-Neumann bracketing for a rather large class of $5$-diagonal real-valued Toeplitz matrices:

\begin{corollary}[$5$-diagonal real-valued Toeplitz matrices]\label{corollary}
Let $h:\mathbb T\to \R$ be the symbol
\beq
	h(x) = a_{2} e^{-2 i x} + a_{1} e^{-i x} + a_0 + a_1 e^{i x} + a_2 e^{2 i x} 
\eeq
where $a_0,a_1,a_2\in\R$ with $a_2> 0$ and $-4\leq \frac {a_1}{a_2}\leq 4$. Then there exist modified Neumann  and Dirichlet boundary conditions, $\mathcal N$ and $\mathcal D$, for the Toeplitz matrix $T_{h.L}$ such that 
\beq\label{thm:D/N-brack}
	 T^{0,\mathcal N}_{h,L_1} \oplus T^{\mathcal N,0}_{h,L_2} 
	 \leq T_{h ,L} 
	 \leq T^{0,\mathcal D}_{h ,L_1} \oplus T^{\mathcal D,0}_{h ,L_2}
\eeq
and
\beq
	\inf_{\mathbb T} h 
	\leq T^{\mathcal N,\mathcal N}_{h,L_1} \oplus T^{\mathcal N,\mathcal N}_{h,L_2} 
	\leq T^{0,\mathcal N}_{h,L_1} \oplus T^{\mathcal N,0}_{h,L_2}
\eeq
for all $L_1,L_2\in\N$ with $L_1+ L_2 = L$ and $L_1,L_2\geq 5$.
\end{corollary}

The upcoming paper \cite{GRM_work_progress} will heavily rely on the established Dirichlet-Neumann bracketing to prove Lifshitz tails of the integrated density of states for self-adjoint Toeplitz matrices with random diagonal perturbations. Fractional powers of Toeplitz matrices of the form \eqref{toeplitz} serve there as model operators. This is a continuation of our study of Lifshitz tails of randomly perturbed fractional Laplacians in \cite{MR4091563}. Generally, Dirichlet-Neumann bracketing is a common tool in proving Lifshitz tails, see e.g. \cite[Sec. 6]{MR2509110}. 
Another main ingredient and of independent interest is a lower bound on the spectral gap above the ground state energy of Toeplitz matrices with modified Neumann boundary condition. We prove here: 

\begin{proposition}[Spectral gap]\label{proposition:main}
Let $n\in\N$, $E_1,..,E_n\in\mathbb T$ be distinct and $\alpha_1,...,\alpha_n\in\N$. Let $g = f_{E_1,\cdots,E_n,\alpha_1,\cdots,\alpha_n}$ be of the form \eqref{def:symbol} and $N=\sum_{i=1}^n \alpha_i$. We denote by $\lambda^L_1\leq ...\leq \lambda^L_L$ the eigenvalues of 
$T_{g,L}^{\mathcal N,\mathcal N}$ counting multiplicities and ordered increasingly. Then $\lambda^L_k= 0$ for $k=1,...,N$ and there exists $C>0$ such that for all $L\geq 2N+1$ 
\beq
	\lambda^L_{N+1}\geq \frac C {L^{2\alpha_{\text{max}}}},
\eeq
where $\alpha_{\text{max}}:=\max\big\{ \alpha_i:\ {i=1,...,n}\big\}$.
\end{proposition}

In the case of $T_{g.L}^{N,N}$, i.e the Neumann boundary conditons defined in \eqref{Neumann}, the latter proposition follows rather directly from the explicit diagonalization of $T_{g.L}^{N,N}$, see \cite{MR1718798}. For  the modified Neumann boundary conditions $T_{g,L}^{\mathcal N,\mathcal N}$ it is more complicated as an explicit diagonalization of $T_{g,L}^{\mathcal N,\mathcal N}$ is not known.

\section{Definition of boundary conditions $\mathcal N$ and $\mathcal D$}

The boundary conditions in Theorem \ref{thm:main1} rely on
 a representation of self-adjoint Toeplitz matrices $T_{f_{E_1,\cdots,E_n,\alpha_1,\cdots,\alpha_n}}$
 as a sum of rank-one operators. To see this we write for $E\in\mathbb T$
 \beq
	 -\Delta_E = D_E^* D_E
 \eeq
 where $D_E := T_{h_E}:\ell^2(\Z)\to\ell^2(\Z)$ is the Laurent matrix given by the symbol $h_E:\mathbb T\to \C$, $h_E(x) = 1 - e^{- i E} e^{-ix}$,  i.e.
 \beq
 D_E = 
	 \begin{pmatrix}
		 \ddots & \ddots  & & & & \\
		0  & 1 & e^{-i E} & & &\\
		 & 0 & 1 & e^{-i E} & &\\
		& & 0  & 1 & e^{-i E}  &\\
		 &  & &\ddots  &\ddots 
	 \end{pmatrix}.
 \eeq
 Using this decomposition and \eqref{toeplitz}, we write
 \beq
 	T_{f_{E_1,\cdots,E_n,\alpha_1,\cdots,\alpha_n}} 
	 = \prod_{i=1}^n \big(D^*_{E_i}D_{E_i}\big)^{\alpha_i}
	 = 
	 \Big( \prod_{i=1}^n D_{E_i}^{\alpha_i}\Big)^*\Big( \prod_{i=1}^n D_{E_i}^{\alpha_i}\Big)
 \eeq
where we used that all Laurent matrices commute. We denote by $(\delta_k)_{k\in\Z}$ the standard basis of $\ell^2(\Z)$. 
Inserting the identity $\mathds 1 = \sum_{k\in\Z} \big|\delta_k\big\>\big\<\delta_k\big|$ in the above, we obtain
 \beq
	 T_{f_{E_1,\cdots,E_n,\alpha_1,\cdots,\alpha_n}} 
	 = \sum_{k\in\Z}\Big|\prod_{i=1}^n D_{E_i}^{\alpha_i}\delta_k\Big\>\Big\< \prod_{i=1}^n D_{E_i}^{\alpha_i}\delta_k\Big|
 \eeq
where the above series converge strongly. For $k\in\Z$ we define the vector 
\beq\label{def:psi}
	\psi^g_k := \prod_{i=1}^n D_{E_i}^{\alpha_i}\delta_k = U_k \prod_{i=1}^n D_{E_i}^{\alpha_i}\delta_0
\eeq
whose support satisfies  $\supp \psi_k^g = [k,k+N]\subset \Z$ where $\supp \varphi = \{n\in\Z: \varphi(n)\neq 0\}$ for $\varphi\in\ell^2(\Z)$ and $N = \sum_{i=1}^n \alpha_i$. In the above $U_k:\ell^2(\Z)\to\ell^2(Z)$, $(U_k x)_n = x_{n-k}$, is the right shift by $k\in\Z$. Summarizing the above computation, we have proved the following:

\begin{proposition}\label{thm:main2}
Let $n\in\N$, $E_1,..,E_n\in\mathbb T$ be distinct and $\alpha_1,...,\alpha_n\in\N$. Let $g = f_{E_1,\cdots,E_n,\alpha_1,\cdots,\alpha_n}$ be of the form \eqref{def:symbol}. Then 
\beq
	T_g = \sum_{k\in\Z} \big|\psi^g_k\big\>\<\psi^g_k\big|
\eeq
with $\psi^g_k\in\ell^2(\Z)$ given by \eqref{def:psi}.
\end{proposition}
Now given Proposition \ref{thm:main2} it is straight forward to define the boundary conditions $\mathcal N$ and $\mathcal D$ in the following way: 

\begin{definition}[Boundary conditions $\mathcal N$ and $\mathcal D$]\label{def:Neumann}
Let $n\in\N$, $E_1,..,E_n\in\mathbb T$ be distinct and $\alpha_1,...,\alpha_n\in\N$. Let $g = f_{E_1,\cdots,E_n,\alpha_1,\cdots,\alpha_n}$ be of the form \eqref{def:symbol}, $N= \sum_{i=1}^n \alpha_i$ and $\psi_k^g$, $k\in\Z$, be given in Proposition \ref{thm:main2}.

For $a\in\Z\cup\{-\infty\}$ and $b\in\Z$ with $b - a > 2N+1$ we define the restriction of $T_g$ to $[a,b]\subset \Z$
with simple boundary conditions at $a$ and 
\begin{enumerate}
\item[(i)]
boundary condition $\mathcal N$ at $b\in \Z$ by
	\beq
	\big(T_g\big)^{0,\mathcal N}_{[a,b]} := \Big( \sum_{\substack{k\in\Z: \\  [k,k+N]\subset (-\infty,b
	] } } \big|\psi_k^g\big\>\big\<\psi_k^g\big| \Big)_{[a,b]}.
	\eeq
 To be precise, for $a=-\infty$ the respective intervals are open at $a$. 
\item[(ii)] 
 boundary condition $\mathcal D$  at $b\in \Z$ by 
	\beq\label{def:dirichlet}
	(T_g)^{0,\mathcal D}_{[a,b]} := 2 (T_g)_{[a,b]} - \big( T_g\big)_{[a,b]}^{0,\mathcal N}.
	\eeq
\end{enumerate}
Accordingly, we define $\big(T_g\big)^{\mathcal N/\mathcal D,0}_{[a,b]}$ by reflection along the anti-diagonal. In particular, 

\begin{enumerate}
\item[(iii)] boundary conditions $\mathcal N$ at both $a,b\in\Z$ are given by
	\beq\label{def:NN}
	\big(T_g\big)^{\mathcal N,\mathcal N}_{[a,b]} := \sum_{\substack{k\in\Z: \\  [k,k+N]\subset [a,b
	] } }   \big|\psi_k^g\big\>\big\<\psi_k^g\big|.
	\eeq
\item [(iv)]
boundary conditions $\mathcal D$ at both $a,b\in\Z$ by 
\beq
	\big(T_g\big)^{\mathcal D,\mathcal D}_{[a,b]} 
	:= 2 (T_g)_{[a,b]} - \big( T_g\big)_{[a,b]}^{\mathcal N,\mathcal N}.
\eeq
\end{enumerate}
\end{definition}

\begin{remarks}
\item From the definition of the boundary conditions $\mathcal N$ and $\mathcal D$ one notes that only the respective $N\times N$ corner of $\big(T_g\big)_{[a,b]}$ at the boundary is changed. More precisely,
	\beq
	\big(T_g\big)^{0,\mathcal N/\mathcal D}_{[a,b]} = \big(T_g\big)_{[a,b]} 
	+
	\begin{pmatrix}
		0  & 0 \\
		0 & \wtilde B_{\mathcal N/\mathcal D}
	\end{pmatrix}
	\eeq
with 
	\beq
	 \wtilde B_{\mathcal N} = - \sum_{\substack{k\in\Z: \\b+1\in [k,k+N]}} P \big|\psi_k^g\big\>\big\<\psi_k^g\big| P \leq 0
	\eeq
and 
	\beq
	 \wtilde B_{\mathcal D} =  \sum_{\substack{k\in\Z:\\ b+1\in [k,k+N]}} P \big|\psi_k^g\big\>\big\<\psi_k^g\big| P \geq 0
	 \eeq
where $P$ is the projection onto the $N$-dimensional space $\ell^2([b-N+1,b])$. 
Therefore $\mathcal N$ and $\mathcal D$ are boundary conditions in the sense of \eqref{def:bc}. 
\item
For functions $g$ as in Theorem \ref{thm:main1}, the latter directly implies  
	\beq
	\big(T_g\big)^{0,\mathcal N}_{[a,b]} \leq \big(T_g\big)_{[a,b]} \leq \big(T_g\big)^{0,\mathcal D}_{[a,b]}.
	\eeq
\end{remarks}

\section{An example}

\begin{example}
Let $E\in \mathbb T$.
We consider the symbol
	\beq
	g(x) = f_{0,E,1,1}(x) = \big( 2- 2 \cos(x)\big)\big( 2- 2 \cos(x-E) \big),\quad x\in \mathbb T. 
	\eeq
 The function $g$ satisfies $g\geq 0 $ and its minimal value is $0$ and attained at  $x=0$ and $x=E$. In the case $E=0$ we have $T_g=(-\Delta)^2$ which was also discussed in the introduction, see \eqref{Lap2}.
A short computation shows that $T_{g} $ is the $5$-diagonal Toeplitz matrix 
	\beq
	T_{g}  = 
	\begin{pmatrix}
		&\ddots	& &  \ddots & &\ddots &\\
		 &	e^{-i E} &	 -2- 2e^{-i E} &   4 + e^{-i E} + e^{iE}       & -2 -2 e^{i E}  & e^{i E} &   \\
		&\ddots	& &  \ddots & & \ddots&
	\end{pmatrix}.
	\eeq
To define the Neumann boundary condition, we write using Proposition \ref{thm:main2}
	\beq
	T_{g} = \sum_{k\in\Z} \big|\psi_k^g\big\>\big\<\psi_k^g\big|
	\eeq
where for $k\in\Z$ we have $\supp \psi_k^g = [k,k+2]$. Moreover 
	\beq
	|\psi_k^g\>\<\psi_k^g|
	= 
	\begin{pmatrix}
		\ddots & & & &  \reflectbox{$\ddots$} \\
		& 1 & -1-e^{i E} & e^{i E} \\
		& -1-e^{-iE} & 2 + e^{-iE}+e^{iE} & -1-e^{iE} \\
		& e^{-iE} & -1 - e^{-iE} & 1 \\
		\reflectbox{$\ddots$} & & & &  \ddots
	\end{pmatrix}
	\eeq
and $0$ everywhere else.
Then the boundary conditions $\mathcal N$ and $\mathcal D$ from Definition~\ref{def:Neumann} are of the  form
	\beq
	\big(T_{g}\big)_{[a,b]}^{\mathcal N,0} = \big(T_{g}\big)_{[a,b]} - 
	\begin{pmatrix}
		3 + e^{i E} + e^{-iE} & -1 - e^{iE} & \cdots \\
		-1 - e^{- iE} & 1 & \\
		\vdots & & \ddots
	\end{pmatrix}
	\eeq
and 
	\beq
	\big(T_{g}\big)_{[a,b]}^{\mathcal D,0} = \big(T_{g}\big)_{[a,b]}  + 
	\begin{pmatrix}
		3 + e^{i E} + e^{-iE} & -1 - e^{iE} & \cdots \\
		-1 - e^{- iE} & 1 & \\
		\vdots & & \ddots
	\end{pmatrix}
	\eeq
where the latter two matrices are $0$ everywhere else. Here one clearly sees that the boundary conditions $\mathcal N$ and $\mathcal D$ consist of adding or subtracting a sign-definite $2\times 2$ matrix in the respective corner of $\big(T_{g}\big)_{[a,b]}$. This is consistent with our definition of boundary conditions in the introduction.
\end{example}

\section{Proof of Theorem \ref{thm:main1} and Corollary \ref{corollary}}

\begin{proof}[Proof of Theorem \ref{thm:main1}]
Let $L_1,L_2\geq 2N+1$ and $L_1+ L_2 = L$. 
The chain of inequalities 
	\beq
	0 = \inf_{\mathbb T} g \leq T^{\mathcal N,\mathcal N}_{g,L_1} \oplus T^{\mathcal N,\mathcal N}_{g,L_2} 
	\leq T^{0,\mathcal N}_{g,L_1} \oplus 	T^{\mathcal N,0}_{g,L_2} \leq T_{g,L} 
	\eeq
follows directly from the definition of the boundary condition $\mathcal N$ as we drop in the definition of $\mathcal N$ non-negative rank-one projections from $T_{g,L}$. For the upper bound in the last inequality of \eqref{thm:D/N-brack} we note that 
	\begin{align}
	T_{g,L} =
	\begin{pmatrix}
		T_{g,L_1} & PT_{g,L} P^\perp \\
		P^\perp T_{g,L} P & T_{g,L_2}
	\end{pmatrix}
	\geq 
	\begin{pmatrix}
		T^{0,\mathcal N}_{g,L_1} & 0 \\
		0 & T^{\mathcal N,0}_{g,L_2}
	\end{pmatrix}
	\end{align}
interpreted as an operator on $\ell^2([1,L_1]) \oplus \ell^2([L_1+1,L])$ and $P$ stands here for projection onto $\ell^2([1,L_1])\oplus \{0\}$ and $P^\perp = \mathds 1- P$. Now Lemma \ref{Neu=Dir} below gives the result as the definition of the modified Dirichlet boundary condition in \eqref{def:dirichlet} is precisely of the form
\eqref{lem:NeqV:eq3}. 
\end{proof}

In the next lemma we show that any boundary condition satisfying the first inequality in \eqref{thm:D/N-brack} naturally induces a boundary condition satisfying the second inequality in \eqref{thm:D/N-brack}.
\begin{lemma}\label{Neu=Dir}
Let $\mathcal H_1$ and $\mathcal H_2$ be two possibly infinite-dimensional Hilbert spaces. Let $A:\mathcal H_1\oplus \mathcal H_2 \to \mathcal H_1\oplus \mathcal H_2$ be a bounded operator 
	\beq
	A = 
	\begin{pmatrix}
		A_{11} & A_{12} \\
		A_{21} & A_{22}
	\end{pmatrix}.
	\eeq
Assume there exist $A_{11}^\mathcal N :\mathcal H_1\to\mathcal H_1$ and $A_{22}^\mathcal N: \mathcal H_2\to\mathcal H_2$ such that
	\beq\label{lem:NeqV:eq1}
	\begin{pmatrix}
		A_{11} & A_{12} \\
		A_{21} & A_{22}
	\end{pmatrix}
	\geq 
	\begin{pmatrix}
		A^\mathcal N_{11} & 0\\
		0 & A^\mathcal N_{22}
	\end{pmatrix}.
	\eeq
Then 
	\beq\label{lem:NeqV:eq3}
	\begin{pmatrix}
		A_{11} & A_{12} \\
		A_{21} & A_{22}
	\end{pmatrix}
	\leq 
	\begin{pmatrix}
		2 A_{11} - A^\mathcal N_{11} & 0\\
		0 & 2 A_{22} - A^\mathcal N_{22}
	\end{pmatrix}.
	\eeq
\end{lemma}

\begin{proof}
We conjugate inequality \eqref{lem:NeqV:eq1} by the unitary $U = \begin{pmatrix} \mathds 1 & 0 \\ 0 &-\mathds 1\end{pmatrix}$. Hence \eqref{lem:NeqV:eq1} is equivalent to 
	\beq\label{lem:NeqV:eq1}
	U^* A U \geq 
	U^* 
	 \begin{pmatrix}
		A^\mathcal N_{11} & 0\\
		0 & A^\mathcal N_{22}
	\end{pmatrix} 
	U 
	=  
	\begin{pmatrix}
		A^\mathcal N_{11} & 0\\
		0 & A^\mathcal N_{22}
	\end{pmatrix}.
	\eeq
We note that 
	\beq
	U^* A U = 
	\begin{pmatrix}
		A_{11} & -A_{12} \\
		-A_{21} & A_{22}
	\end{pmatrix}
	\eeq
which together with \eqref{lem:NeqV:eq1} gives
	\beq
	\begin{pmatrix}
		2 A_{11} & 0\\
		0 & 2 A_{22}
	\end{pmatrix}
	-
	\begin{pmatrix}
		A_{11} & -A_{12} \\
		-A_{21} & A_{22}
	\end{pmatrix}
	\leq 
	\begin{pmatrix}
		2 A_{11} & 0\\
		0 & 2 A_{22}
	\end{pmatrix}
	-
	\begin{pmatrix}
		A^\mathcal N_{11} & 0\\
		0 & A^\mathcal N_{22}
	\end{pmatrix}
	\eeq
which is the result.
\end{proof}

\begin{proof}[Proof of Corollary \ref{corollary}]
Let $b\in \mathbb T$. We compute for $x\in\mathbb T$
	\begin{align}\label{pf:cor:1}
	&(2-2 \cos(x-b))(2-2 \cos(x+b)) \notag\\
	=& (2 - e^{i x} e^{-i b} - e^{-i x} e^{i b})(2 - e^{i x} e^{i b} - e^{-i x} e^{-i b})\notag \\
	=& e^{-2i x} - 4 \cos(b) e^{-i x}  + 4 + 2 \cos(2b) - 4 \cos(b) e^{i x} + e^{2ix} = : w_b(x). 
	\end{align}
Let $h:\mathbb T\to\R$ be of the form as described in Corollary \ref{corollary}
	\beq
	h(x) = a_{2} e^{-2 i x} + a_{1} e^{-i x} + a_0 + a_1 e^{i x} + a_2 e^{2 i x} 
	\eeq
where $a_0,a_1,a_2\in\R$ with $a_2> 0$ and $-4\leq \frac {a_1}{a_2}\leq 4$. We rewrite
	\beq
	h(x) = a_2 \big(  e^{-2 i x} + \frac{a_{1}}{a_2} e^{-i x} 
	+ \frac{a_0}{a_2} + \frac{a_1}{a_2} e^{i x} + e^{2 i x} \big).
	\eeq
As we assumed $-4\leq\frac{a_0}{a_2}\leq 4$ there exists $b\in\mathbb T$ such that $4\cos(b) = \frac{a_1}{a_2}$ and hence using the definition of $w_b$ in \eqref{pf:cor:1} we obtain
	\begin{align}
	h(x) 
	&= a_2\big( (2-2 \cos(x-b))(2-2 \cos(x+b)) \big) + a_0 -  4 - 2 \cos(2b) \notag\\
	&=a_2 w_b(x) + c
	\end{align}
with $c:= a_0 -  4 - 2 \cos(2b)$. 
Theorem \ref{thm:main1} implies there exists boundary conditions $\mathcal N$ and $\mathcal D$ such that
	\beq\label{cor:pf:D/N-brack}
	0 = \inf_{\mathbb T} w_b 
	\leq T^{\mathcal N,\mathcal N}_{w_b,L_1} \oplus T^{\mathcal N,\mathcal N}_{w_b,L_2} 
	\leq T^{0,\mathcal N}_{w_b,L_1} \oplus T^{\mathcal N,0}_{w_b,L_2} \leq T_{w_b ,L} 
	\leq T^{0,\mathcal D}_{w_b ,L_1} \oplus T^{\mathcal D,0}_{w_b ,L_2}
	\eeq
for all $L_1,L_2\in\N$ with $L_1+ L_2 = L$ and $L_1,L_2\geq 2N+1$. Multiplying $w_b$ with $a_2\geq 0$ and adding $c$ will not change the chain of operator inequalities \eqref{cor:pf:D/N-brack} and the result follows. 
\end{proof}

\section{Proof of Proposition \ref{proposition:main}}

\begin{proof}[Proof of Proposition \ref{proposition:main}]
Fix $g$ and $N$ as in the assumptions and $L\in\N$ with $L\geq 2N+1$. 
We first prove that  $\lambda^L_k = 0$ for all $k = 1,...,N$. To do so, we consider the $N$ vectors
	\beq\label{eq:vectors}
	\varphi_{E_i}^{j_i} =\big(k^{j_i} e^{i E_i k}\big)_{k=1,..,L} 
	= \big( 1^{j_i}e^{i E_i }, \cdots, L^{j_i}e^{i E_i L}\big)^T \in \C^L = \ell^2([1,L])
	\eeq 
where $i=1,...,n$ and $j_i = 0,...,\alpha_i-1$. A computation shows that for all $k\in 1,...,L-N$
	\beq
	\big( D_{E_i}^{\alpha_i} \varphi_{E_i}^{j_i}\big)_k = 0
	\eeq
where we see $D_{E_i}^{\alpha_i}$ here as an operator $D_{E_i}^{\alpha_i}: \ell^2([1,L])\to \ell^2([1,L])$. 
Therefore by the definition of $\psi_k^g$ in \eqref{def:psi} we obtain for $k\in 1,...,L-N$
	\beq
	\big\<\psi_k^g, \varphi_{E_i}^{j_i}\big\>= 0
	\eeq
for all $i=1,...,n$ and $j_i = 0,...,\alpha_i-1$.
Recalling the definition of $T_{g,L}^{\mathcal N,\mathcal N}$ in \eqref{def:NN}, we obtain from the previous identity
	\beq
	T_{g,L}^{\mathcal N,\mathcal N} \varphi_{E_i}^{j_i} 
	= \sum_{\substack{m\in\Z:\\ [k,k+N]\subset [1,L]}} \big|\psi_k^g\big\>\big\<\psi_k^g\big|\,\varphi_{E_i}^{j_i}\>= 0
	\eeq
for $i=1,...,n$ and $j_i = 0,...,\alpha_i-1$. Lemma \ref{lem:linIndep} shows that the $N$ vectors in \eqref{eq:vectors} are linearly independent and therefore span a $N$ dimensional space which implies $\lambda^L_k = 0$ for $k = 1,...,N$. 

Next we prove the lower bound on $\lambda^L_{N+1}= \lambda^L_{N+1}\big(T_{g,L}^{\mathcal N,\mathcal N}\big)$, where we use the notation $\lambda^L_k (\cdot)$ if we want to emphasize to underlying operator. We consider first the $L\times L$ restriction of $T_g$ with periodic boundary conditions
	\beq
	T_{g,L}^{\text{per}} 
	:= 
	\begin{pmatrix}
		a_0  & \cdots &a_N &    & &    & a_{-N} &\cdots &a_{-1}  \\
		     & \ddots  &   &  \ddots       &              &   &               &   \ddots               &      \vdots      \\
		     &  & \ddots   &        &    \ddots           &                &      &     &    a_{-1}               \\
		     &   &  a_{-N} &  \cdots       &      a_0  &      \cdots          & a_N      &     &                   \\
			  a_{1}   &  &     &        &   \ddots     &        &      \ddots     &     &                   \\
		 	\vdots   & \ddots &   &       &        &      \ddots          &      & \ddots    &                   \\
		a_1&\cdots	& a_{N-1}  &  &  &  &     a_{-N} & \cdots   & a_{0}\\
	\end{pmatrix}.
	\eeq
 For $k=1,...,L$ we define the vector
$
\psi^{(k)} = \big(\psi_1^{(k)},.., \psi_L^{(k)}\big)^T\in \C^L
$
	\beq
	\psi_m^{(k)}:= \frac 1 {\sqrt L} e^{\frac{2\pi k (m-1) } L i},\qquad m=1,...,L. 
	\eeq
A computation shows for $k=1,...,L$ that
	\beq\label{pf:ev:per}
	T_{g,L}^{\text{per}} \psi^{(k)} = g\Big( \frac{2\pi k} L\Big) \psi^{(k)}.
	\eeq
Therefore, the family of vectors $\big( \psi^{(k)}\big)_{k=1,..,L}$ form an ONB of eigenvectors of $T_{g,L}^{\text{per}}$
corresponding to the eigenvalues $g\Big( \frac{2\pi k} L\Big)$, $k=1,...,L$. 

Using the definition of $T_{g,L}^{\mathcal N,\mathcal N}$ in \eqref{def:NN}, we observe that
	\beq\label{prop:pf:eq0}
	T_{g,L}^{\text{per}}  - T_{g,L}^{\mathcal N,\mathcal N} = \sum_{k=L-N+1}^L |\psi_k^g\>\<\psi_k^g|
	\eeq
where for $k=L-N+1,...,L$
	\beq
	\psi_k^g = \big(c_{L-k+1}, \cdots, c_N, 0, \cdots, 0,  c_{0}, \cdots, c_{L-k} \big)^T\in\C^L
	\eeq
with $c_k := (\psi_0^g)_k$ for $k=0,...,N$.
Therefore, the difference in \eqref{prop:pf:eq0} is rank $N$.  
From the first part of the proof we know that  $\lambda^L_j(T_{g,L}^{\mathcal N,\mathcal N}) = 0 $ for $j=1,...,N$. Now the min-max principle implies the lower bound
	\beq\label{prop:pf:eq1}
	\lambda^L_{N+1}\big(T_{g,L}^{\mathcal N,\mathcal N}\big) 
	\geq \lambda^L_1\big(T_{g,L}^{\text{per}}\big) 
	= \min_{k=1,...,L} g\Big( \frac{2\pi k} 	L\Big)
	\eeq
and the last equality follows from \eqref{pf:ev:per}. 
Next we define for $E\in \mathbb T$ the unitary $U_E:\C^L\to\C^L$, $(U_E b)_m = e^{-i E m} b_m$ for $b\in \C^L$ and $m=1,...,L$. Then, by the definition of $\psi_k^g$ the following identity holds 
	\beq
	U_ET_{g,L}^{\mathcal N,\mathcal N} U_E^* 
	= T_{g_E,L}^{\mathcal N,\mathcal N}
	\eeq
where $g_E(x) = g(x-E)$, $x\in\mathbb T$, and we extended $g$ here periodically such that $g(x-E)$ makes sense for any $x\in\mathbb T$ and $E\in \mathbb T$. As the spectrum does not change under conjugation by a unitary, we obtain
	\beq
	\lambda^L_{N+1}\big(T_{g,L}^{\mathcal N,\mathcal N}\big)  
	= \lambda^L_{N+1}\big(T_{g_E,L}^{\mathcal N,\mathcal N}\big)
	\eeq
for all $E\in \mathbb T$ and using the lower bound \eqref{prop:pf:eq1} we end up with
	\begin{align}\label{prop:pf:eq2}
	\lambda^L_{N+1}\big(T_{g,L}^{\mathcal N,\mathcal N}\big) 
	& = \max_{E\in \mathbb T}\ \lambda^L_{N+1}\big(T_{g_E,L}^{\mathcal N,\mathcal N}\big) \notag\\
	& \geq
	\max_{E\in \mathbb T}
	\min_{k=1,...,L} g_E\Big( \frac{2\pi k} L\Big).
	\end{align}
Given the distinct minima $E_1,...,E_n\in\mathbb T$  of the function $g$, Lemma \ref{lm:distance} below provides a constant $C_1>0$ such that for all $L>2N+1$ there exists $\tilde E\in \mathbb T$ such that
	\beq\label{inequality}
	\frac {C_1} L \leq \min_{i=1,...,n} \dist\big(E_i, \Big(\frac{2\pi k} L- \tilde E\Big)\text{mod}\, 2\pi: k = 1,...,L\big) \leq \frac {\pi } L.
	\eeq
We note that $C_1>0$ in the above is independent of $L$ and only depends on $n$. 
Since $E_1,...,E_n$ are the minima of the function $g$, we obtain with the $\tilde E\in \mathbb T$ found above, inequality \eqref{inequality} and Taylor's theorem the lower bound
	\begin{align}
	\eqref{prop:pf:eq2}
	&\geq \min_{k=1,...,L} g\Big(\frac{2\pi k} L - \tilde E\Big)\notag\\
	& \geq \frac {C_2}{L^{2\alpha_{\text{max}}}}
	\end{align}
for some $C_2>0$ depending on $g$ but independently of $L$ which is the assertion.
\end{proof}

\begin{lemma}\label{lm:distance}
Let $E_1,...,E_n\in \mathbb T$ be $n\in\N$ distinct points and set $\mathcal E^{(n)} := \big\{E_i,\ i=1,...,n\big\}$. Then there exists $\wtilde E\in \mathbb T$ such that  
	\beq\label{lm:dist}
	\dist(\mathcal S^{\wtilde E}_L, \mathcal E^{(n)}) \geq \frac {2\pi} {2^n} \frac 1 L
	\eeq
where
	\beq
	\mathcal S^{\wtilde E}_L := \Big\{  \Big(\frac{2\pi k} L- \wtilde E\Big)\text{mod}\, 2\pi: k = 1,...,L  \Big\}
	\eeq
 and $\dist(A,B)=\min\big\{ |a-b|:\ a\in A,b\in B\big\}$ for $A,B\subset\R$.
\end{lemma}

\begin{proof}
We prove the lemma by induction on $n\in\N$. 

For $n=1$ let $\mathcal E^{(1)}=\big\{E_1\big\}$. Then  we choose $\wtilde E = -  E_1 + \frac{\pi} L$ 
and therefore \eqref{lm:dist} is true. 

Assume the result is true for $n-1$ distinct points $E_1,...,E_{n-1}$ and let $E_n$ be a point distinct from the others. By assumption there exists $\wtilde E$ such that 
	\beq
	\dist(\mathcal S^{\wtilde E}_L, \mathcal E^{(n-1)}) \geq \frac {2\pi} {2^{n-1}} \frac 1 L.
	\eeq
If $\dist\big(\mathcal S^{\wtilde E}_L, E_n\big) \geq \frac {2\pi} {2^{n-1}} \frac 1 L$ we are done. If this is not the case we obtain by adding or subtracting $ \frac {2\pi} {2^{n}} \frac 1 L$ to $\wtilde E$ that there exists  $\hat E\in \mathbb T$ such that 
	\beq\label{lm:dist_pf:eq1}
	\dist\big(\mathcal S^{\hat E}_L, E_n \big) \geq \frac {2\pi} {2^{n}} \frac 1 L.
	\eeq
Since $|\wtilde E -\hat E|\leq \frac {2\pi} {2^{n}} \frac 1 L$ and $\dist(\mathcal S^{\tilde E}_L, \mathcal E^{(n-1)}) \geq \frac {2\pi} {2^{n-1}} \frac 1 L$ , we obtain
	\beq
	\dist(\mathcal S^{\hat E}_L, \mathcal E^{(n-1)}) \geq \frac {2\pi} {2^{n}} \frac 1 L
	\eeq
which is the assertion together with \eqref{lm:dist_pf:eq1}. 
\end{proof}

\begin{lemma}\label{lem:linIndep}
Let $n\in\N$, $E_1,..,E_n\in\mathbb T$ be distinct, $\alpha_1,...,\alpha_n\in\N$ and $N=\sum_{i=1}^n\alpha_i$. Moreover, let $L\in\N$ with $L\geq N$. The $N$ vectors
	\beq\label{eq:vectors2}
	\varphi_{E_i}^{j_i} =\big(k^{j_i} e^{i E_i k}\big)_{k=1,..,L} 
	= \big( 1^{j_i}e^{i E_i }, \cdots, L^{j_i}e^{i E_i L}\big)^T \in \C^L
	\eeq 
where $i=1,...,n$ and $j_i = 0,...,\alpha_i-1$ are linearly independent.
\end{lemma}

\begin{proof}
Let $i\in\{1,...,n\}$ and $j_i\in\{0,...,\alpha_i-1\}$. We introduce the short-hand notation
$
	z_i := e^{i E_i}
$
and define the truncation of $\varphi_i^{j_i}$ to $\C^N$
	\beq
	\hat \varphi_i^{j_i} := \big( z_i, 2^{j_i} z_i^2, ..., N^{j_i} z^{N}_i\big)^T \in \C^N.
	\eeq
This is just the truncation of $\varphi_i^{j_i}$ to the first $N$ rows. Now
$
\det
\begin{pmatrix}
\hat \varphi^0_1, ... ,\hat \varphi_1^{\alpha_1-1},\hat \varphi_2^0,\cdots ,\hat\varphi_n^{\alpha_n-1} 
\end{pmatrix} 
$
 is a confluent Vandermonde determinant which can be computed explicitly and evaluates to
	\beq\label{det}
	\left|\det
	\begin{pmatrix}
	\hat \varphi^0_1, ... ,\hat \varphi_1^{\alpha_1-1},\hat \varphi_2^0,\cdots ,\hat\varphi_n^{\alpha_n-1} 
	\end{pmatrix} 
	\right|
	= 
	\prod_{i=1}^n (\alpha_i-1)! \prod_{1\leq i<j\leq n} \big| z_i - z_j|^{\alpha_i \alpha_j}, 
	\eeq
see e.g \cite[Thm. 1]{MR568779}. Since $z_i\neq z_j$ for all $i\neq j$, we obtain that the latter determinant is non-zero. Therefore, the $N$ vectors $\hat\varphi_1^0,...,\hat\varphi_n^{\alpha_n-1}$ are linearly independent. This implies that the vectors $\big\{\varphi_i^{j_i}: i=1,...,n, j_i=0,...,\alpha_i-1\big\}$ are linearly independent as well.
\end{proof}

\section*{Acknowledgements}
The author thanks Constanza Rojas-Molina for many interesting and enjoyable discussions on the subject and Peter M\"uller for helpful comments on an earlier version of the paper.


\begin{thebibliography}{GRM20}
\providecommand{\url}[1]{{\tt #1}}
\providecommand{\urlprefix}{URL }
\providecommand{\eprint}[2][]{e-print {#2}}

\bibitem[BG05]{MR2179973}
A.~B\"{o}ttcher and S.~M. Grudsky, {\em Spectral properties of banded
  {T}oeplitz matrices\/}, Society for Industrial and Applied Mathematics
  (SIAM), Philadelphia, PA, 2005.

\bibitem[GRM]{GRM_work_progress}
M.~Gebert and C.~Rojas-Molina, Lifshitz tails for random diagonal perturbations
  of Toeplitz matrices, \textit{in preparation}.

\bibitem[GRM20]{MR4091563}
M.~Gebert and C.~Rojas-Molina, Lifshitz tails for the fractional {A}nderson
  model, {\em J. Stat. Phys.\/} {\bf 179}, 341--353 (2020).

\bibitem[HG80]{MR568779}
T.~T. Ha and J.~A. Gibson, A note on the determinant of a functional confluent
  {V}andermonde matrix and controllability, {\em Linear Algebra Appl.\/} {\bf
  30}, 69--75 (1980).

\bibitem[Kir08]{MR2509110}
W.~Kirsch, An invitation to random {S}chr\"{o}dinger operators, Panoramas et Synth\`eses \textbf{25}, 1--119 (2008).

\bibitem[KM07]{MR2307751}
W.~Kirsch and B.~Metzger, The integrated density of states for random
  {S}chr\"{o}dinger operators,
  Proc. Sympos. Pure Math., vol.~76, Amer. Math. Soc., Providence, RI, 2007,
  pp. 649--696.

\bibitem[NCT99]{MR1718798}
M.~K. Ng, R.~H. Chan and W.-C. Tang, A fast algorithm for deblurring models
  with {N}eumann boundary conditions, {\em SIAM J. Sci. Comput.\/} {\bf 21},
  851--866 (1999).

\bibitem[RS78]{MR0493421}
M.~Reed and B.~Simon, {\em Methods of modern mathematical physics. {IV}.
  {A}nalysis of operators\/}, Academic Press, New York, 1978.

\bibitem[Sim85]{MR784931}
B.~Simon, Lifschitz tails for the {A}nderson model, {\em J. Stat. Phys.\/}
  {\bf 38}, 65--76 (1985).

\end{thebibliography}
\end{document}